\documentclass{amsart}

\usepackage{enumerate}
\usepackage{amsmath,amsfonts,amssymb,amsthm}
\usepackage{csquotes}
\usepackage{graphicx,amsaddr}

\newtheorem{thm}{Theorem}[section]
\newtheorem{lem}[thm]{Lemma}
\newtheorem{prop}[thm]{Proposition}
\newtheorem{cor}[thm]{Corollary}

\theoremstyle{definition}
\newtheorem{defn}[thm]{Definition}
\newtheorem{exam}[thm]{Example}

\theoremstyle{remark}
\newtheorem{rem}[thm]{Remark}

\numberwithin{equation}{section}

\title[H. V. Dedania and J. G. Patel*]{Uniqueness of Norm and Faithfulness of some product Banach Algebras}
\author{H. V. Dedania}

\address{Dept. of Mathematics, Sardar Patel University, Vallabh Vidyanagar 388120, Gujarat, India}
\email{hvdedania@gmail.com}

\author{J. G. Patel*}

\address{Dept. of Mathematics, Sardar Patel University, Vallabh Vidyanagar 388120, Gujarat, India}
\email{jatinprofessor39@gmail.com}

\begin{document}
	
	
	
\begin{abstract}
		We prove that the faithfull and uniqueness of norm properties are stable in different product algebras such as direct-sum product algebra, convolution product algebra, and module product algebra. Further, we exhibit that these properties are not stable in null product algebra, and also give a common sufficient condition in terms of algebra norm for the co-dimension of $\mathcal{A}^2 = \text{span} \{ ab : a,b \in \mathcal{A}\}$ to be finite in $\mathcal{A}$ and $\mathcal{A}^{2} = \mathcal{A} \ ( \text{when } \overline{\mathcal{A}^2} = \mathcal{A})$.
\end{abstract}
\maketitle
	
\section{Introduction}
	
	Throughout $\mathcal{A}$ is an (associative) algebra over the complex field $\mathbb{C}$. A norm $\| \cdot \|$ on $\mathcal{A}$, means it is a linear norm and submultiplicative, i.e. $\| a b \| \leq \|a \| \| b \| \ (a, b \in \mathcal{A})$. If $\| \cdot \|$ gives a complete metric topology, then we say $\mathcal{A}$ is a Banach algebra. The sets $N(\mathcal{A})$ and $N_{c}(\mathcal{A})$  denote the set of all algebra norms (up to equivalent), the set of all Banach algebra norms (up to equivalent), respectively. Let $\mathcal{I}$ be an ideal in an algebra $\mathcal{A}$. Then $\mathcal{A}$ is faithful over $\mathcal{I}$ if $ab=0 \ (b \in \mathcal{I})$, then $a=0$. An algebra $\mathcal{A}$ is faithful if it is faithful over itself~\cite{Da:00}. It is shown in~\cite{DP:22(a)} that the uniqueness of norm property on the cartesian product algebra $\mathcal{A} \times \mathcal{B}$ depends on the uniqueness of norm on $\mathcal{A}$ and $\mathcal{B}$, where $\mathcal{A}$ and $\mathcal{B}$ are Banach algebras. So, it is natural to study the uniqueness of norm on other product algebras namely, the direct-sum product algebra [Definition \ref{Def:Directsum}], the convolution product algebra [Definition \ref{Defn:Convolution}] and the module product algebra [Definition \ref{Def:Module}]. In some cases, we obtained necessary and sufficient conditions for uniqueness of norm and faithfulness.
	
	Dales and Loy have supplied an interesting example that it has two inequivalent complete norms~\cite{DL:97}. We proved in~\cite{DP:22(a)}, \cite{DP:22(b)} that it has infinitely many complete norms and incomplete norms. In this paper, we generalised this result for a normed algebra $\mathcal{A}$ with some condition on $\mathcal{A}$. We also note down that the condition ``$N(\mathcal{A}) = N_c(\mathcal{A})$ contains only one element" served as a sufficient condition for the co-dimension of $\mathcal{A}^2$ to be finite in $\mathcal{A}$, and $\mathcal{A}^2 = \mathcal{A}$ (in the case of $\mathcal{A}^2$ is dense in $\mathcal{A}$).

\section{Main Results}
	
\begin{prop}\cite[P. 14]{Wa:14}\label{1}
		Let $\mathcal A$ be a Banach algebra, and let $\mathcal{I}_1$, $\mathcal{I}_2$ be closed ideals in $\mathcal{A}$ with $\mathcal{I}_1 \subset \mathcal{I}_2$. Suppose that both $\mathcal{A}/ \mathcal{I}_2$ and $\mathcal{I}_2/\mathcal{I}_1$ have unique algebra norms. Then $\mathcal{A}/\mathcal{I}_1$ has a unique algebra norm.
\end{prop}

\begin{defn}\cite{Ka:16}\label{Def:Directsum}
	Let $\mathcal{A}$ be an algebra and $\mathcal B$ be a subalgebra of $\mathcal{A}$. Then $\mathcal{A} \times {\mathcal B}$ is an algebra with pointwise linear operations and direct-sum product `$\times_d$' define as $(a,b)\times_d (c,d) = (ac+ad+bc, bd) \ ((a,b), (c,d) \in \mathcal{A} \times_d {\mathcal B})$. 
	The direct-sum product is a generalization of the product defined on $\mathcal{A}_e$ (means unitization of $\mathcal{A}$). Further, $\|(a, b)\|_1 = \| a \|+\| b \|$ and $|(a, b)| = \max\{\|a-b\|, \|b\|\} \ ((a,b) \in \mathcal{A} \times_d \mathcal{B})$ are algebra norms on $\mathcal{A} \times_d {\mathcal B}$.
\end{defn} 

\begin{rem}
(i) Let $(\mathcal{A}, \| \cdot \|)$ be a Banach algebra and $\mathcal{B}$ be a closed subalgebra of $\mathcal{A}$. Then $(\mathcal{A} \times_d {\mathcal B}, \| \cdot \|_1)$ is a Banach algebra.\\
(ii) $\mathcal{A} \times \{0\}$ is a closed ideal of $\mathcal{A} \times_d \mathcal{B}$.\\
(iii) By the first isomorphism theorem, $(\mathcal{A} \times_d \mathcal{B})/\mathcal{A} \cong \mathcal{B}$.\\
(iv) $\| \cdot \|_1$ and $|\cdot|$ are equivalent on $\mathcal{A} \times_d {\mathcal B}$.
\end{rem}

\begin{lem}
	Let $\mathcal{A}$ be an algebra and $\mathcal{B}$ be a subalgebra of $\mathcal{A}$. Then $\mathcal{A}$ and $\mathcal{B}$ are faithful if and only if $\mathcal{A} \times_d \mathcal{B}$ is faithful.
\end{lem}

\begin{proof}
	Assume that $\mathcal{A}$ and $\mathcal{B}$ are faithful. Let $(a,b) \in \mathcal{A} \times_d \mathcal{B}$ with $(a,b) \times_d (c,d)=(0,0) \ ((c,d) \in \mathcal{A} \times_d \mathcal{B})$. Then $(ac+ad+bc, bd)=(0,0) \; ((c,d) \in \mathcal{A} \times_d \mathcal{B})$. That implies $b=0$ as $\mathcal{B}$ is faithful. Put $b=0$ in $ac+ad+bc=0$, then $a(c+d)=0$ as $d$ was an arbitrary element, we can put $d=0$. That gives $a=0$ as $\mathcal{A}$ is faithful. Therefore $\mathcal{A} \times_d \mathcal{B}$ is faithful. Conversely,  assume that $\mathcal{A} \times_d \mathcal{B}$ is faithful. Suppose, if possible that there exists $0 \neq a \in \mathcal{A}$ such that $a x =0 \; (x \in \mathcal{A})$. Set $(a,0) \in \mathcal{A} \times_d \mathcal{B}$ and let $(x,y) \in \mathcal{A} \times_d \mathcal{B}$. Then $(a,0)\times_d (x,y) = (ax+ay+0 x, 0 y) = (0,0)$, which gives contradiction to faithfulness of $\mathcal{A} \times_d \mathcal{B}$. There fore $\mathcal{A}$ is faithful. Similarly, suppose $\mathcal{B}$ is not faithful, there exists $0 \neq b \in \mathcal{B}$ such that $b y =0 \; \; (y \in \mathcal{B})$. Set $(-b,b) \in \mathcal{A} \times_d \mathcal{B}$ and let $(x,y) \in \mathcal{A} \times_d \mathcal{B}$. Then $(-b, b) \times_d (x,y)=(-bx-by+bx, by)= (-by, by)=(0,0)$, which gives a contradiction to faithfulness of $\mathcal{A} \times_d \mathcal{B}$. Hence $\mathcal{B}$ is faithful.
\end{proof}

\begin{thm}\label{uniqueness norm on direct-sum product}
	Let $(\mathcal{A}, \| \cdot \|)$ be a Banach algebra and $\mathcal B$ be a closed subalgebra of $\mathcal{A}$. Then $N(\mathcal{A})$ and $N(\mathcal{B})$ are singleton if and only if $N(\mathcal{A} \times_d \mathcal{B})$ is singleton.
\end{thm}

\begin{proof}
	Assume that $N(\mathcal{A})$ and $N(\mathcal B)$ are singleton. Consider $J_1=\{(0,0)\}$ and $J_2= \mathcal{A} \times \{0\}$. Then $J_1$ and $J_2$ are closed ideals in $\mathcal{A} \times_d \mathcal{B}, \, J_1 \subset J_2, \, (\mathcal{A} \times_d \mathcal{B})/J_2 \cong \mathcal{B}$, and $J_2/J_1 \cong \mathcal{A}$. Since $N(\mathcal{A})$ and $N(\mathcal{B})$ are singleton, $N((\mathcal{A} \times_d \mathcal{B})/J_2)$ and $N(J_2/ J_1)$ are singleton. Hence, by Proposition~\ref{1}, $\mathcal{A} \times_d \mathcal B$ has the unique algebra norm, i.e., $N(\mathcal{A} \times_d \mathcal B)$ is singleton.
	Conversely, assume that $N(\mathcal{A} \times \mathcal{B})$ is singleton. Suppose, if possible, $N(\mathcal{A})$ is not singleton. Choose two norms $p_1(\cdot), p_2(\cdot) \in N(\mathcal{A})$ such that $\widetilde{p_1}(\cdot)$ and $\widetilde{p_2}(\cdot)$ are two different classes in $N(\mathcal{A})$. Then $p_1(\cdot)$ and $p_2(\cdot)$ are not equivalent on $\mathcal{A}$. Next take $q(\cdot) \in N(\mathcal{B})$.  Define $q_1(x,y)=p_1(x)+q(y)$ and $q_2(x,y)=p_2(x)+q(y)$ for $(x,y) \in \mathcal{A} \times \mathcal{B}$. Then $q_1(\cdot)$ and $q_2(\cdot)$ are non-equivalent norms on $\mathcal{A} \times \mathcal{B}$, and hence $N(\mathcal{A} \times \mathcal{B})$ is not singleton, which is a contradiction. Thus $N(\mathcal{A})$ must be singleton. Similarly, $N(\mathcal{B})$ is singleton. 
\end{proof}

\begin{defn}\cite{Ka:16}\label{Defn:Convolution}
	Let $\mathcal{A}$ be an algebra and $\mathcal{I}$ be an ideal of $\mathcal{A}$. Then $\mathcal{A} \times \mathcal{I}$ is an algebra with co-ordinatewise linear operations and the convolution product `$\times_c$' define as $(a,x) \times_c (b,y) = (ab+xy, ay+xb) \  ((a,x), (b,y) \in \mathcal{A} \times_c \mathcal{I})$. Then $\mathcal{A} \times_c \mathcal{I}$ is called a \emph{convolution product} algebra.
\end{defn}

\begin{rem}
	(i) Let $\mathcal{I}$ be a closed ideal in $\mathcal{A}$. Then $\mathcal{I} \times_c \mathcal{I}$ is a closed ideal in $\mathcal{A} \times_c \mathcal{I}$.\\
	(ii) By the first isomorphism theorem, $(\mathcal{A} \times_c \mathcal{I})/(\mathcal{I} \times_c \mathcal{I}) \cong \mathcal{A}/\mathcal{I}$.
\end{rem}

\begin{lem}
	Let $\mathcal{A}$ be an algebra and $\mathcal{I}$ be an ideal in  $\mathcal{A}$. Then $\mathcal{A}$ is faithful if and only if $\mathcal{A} \times_c \mathcal{I}$ is faithful.
\end{lem}

\begin{proof}
	Assume that $\mathcal{A}$ is faithful. Let $(a,x) \in \mathcal{A} \times_c \mathcal{I}$ with $(a,x) \times_c (b,x)=(0,0) \ ((b,x) \in \mathcal{A} \times_c \mathcal{I})$. Then $(a,x) \times_c (b,0)= (ab, xb) = (0,0) \ ((b,0) \in \mathcal{A} \times_c \mathcal{I})$, $ab=xb=0 \ (b \in \mathcal{A})$. Hence $a=x=0$ as $\mathcal{A}$ is faithful.  Conversely,  assume that $\mathcal{A} \times_c \mathcal{I}$ is faithful. Suppose, if possible, that there exists $0 \neq a \in \mathcal{A}$ such that $a x =0 \ (x \in \mathcal{A})$. Set $(a,0) \in \mathcal{A} \times_c \mathcal{I}$ and let $(x,y) \in \mathcal{A} \times_c \mathcal{I}$. Then $(a,0)\times_c (b,y) = (ab, ay) = (0,0)$, which gives a contradiction to faithfulness of $\mathcal{A} \times_c \mathcal{I}$. Hence $\mathcal{A}$ is faithful.
\end{proof}

\begin{lem}
	Let $\mathcal{A}$ be an algebra and $\mathcal{I}$ be a closed ideal of $\mathcal{A}$. Then $N(\mathcal{A} \times_c \mathcal{I})$ is singleton implies $N(\mathcal{A})$ is singleton.
\end{lem}

\begin{proof}
	The proof is similar to the ``converse part" of proof of  Theorem~\ref{uniqueness norm on direct-sum product}.
\end{proof}

\begin{thm}
Let $\mathcal{A}$ be a Banach algebra and $\mathcal{I}$ be a closed ideal of $\mathcal{A}$. Then $N(\mathcal{A}/ \mathcal{I})$ and $N(\mathcal{I} \times
_c \mathcal{I})$ are singleton implies $N(\mathcal{A} \times_c \mathcal{I})$ is singleton.
\end{thm}

\begin{proof}
	Set $J_1 = \{0\} \times \{0\}$ and $J_2 = \mathcal{I} \times \mathcal{I}$. Then $J_1$ and $J_2$ are two closed ideals of $\mathcal{A}, \, (\mathcal{A} \times_c \mathcal{I}/ J_2) \cong \mathcal{A}/ \mathcal{I}$ and $J_2/ J_1 \cong J_2$. Since $N(\mathcal{A}/ \mathcal{I})$ and $N(\mathcal{I} \times
	_c \mathcal{I})$ are singleton, $N(\mathcal{A} \times_c \mathcal{I}/ J_2)$ and $N(J_2/ J_1)$ are singleton. Hence, by Proposition~\ref{1}, $N(\mathcal{A} \times_c \mathcal{I})$ is singleton.
\end{proof}

\begin{defn}\cite{Da:00}\label{Def:Module}
	Let $(\mathcal{A}, \| \cdot \|)$ be a Banach algebra and $(\mathcal{X}, | \cdot |)$ be a Banach $\mathcal{A}$-bimodule. For $(a,x), (b,y) \in \mathcal{A} \times \mathcal{X}$, define $(a,x) \times_{m} (b,y) = (ab, ay+xb)$. Then $(\mathcal{A} \times_{m} \mathcal{X}, \times_{m})$ is called the module product algebra. It is a Banach algebra with the norm $\|(a,x) \|_1 = \|a\| + |x| \ ((a,x) \in \mathcal{A} \times_{m} \mathcal{X})$.
\end{defn}

\begin{rem}
	(i) $\{0\} \times \mathcal{X}$ is a closed ideal in $\mathcal{A} \times_{m} \mathcal{X}$.\\
	(ii) By the first isomorphism theorem, $(\mathcal{A} \times_{m} \mathcal{X})/ (\{0\} \times \mathcal{X}) \cong \mathcal{A}$.
\end{rem}

\begin{defn}
	Let $\mathcal{X}$ be an $\mathcal{A}$-bimodule. It is said to be $\mathcal{A}$-faithful if $xa = 0 \ (a \in \mathcal{A})$, then $x =0$.
\end{defn}

\begin{lem}\label{Lem:Faithful property on Module product algebra}
	If $\mathcal{A}$ is faithful and $\mathcal{X}$ is $\mathcal{A}$-faithful, then $\mathcal{A} \times_m \mathcal{X}$ is faithful.
\end{lem}

\begin{proof}
	Let $(a,x) \in \mathcal{A} \times_m \mathcal{X}$ such that $(a,x) \times_{m} (b,y) = (0,0) \ ((b,y) \in \mathcal{A} \times_m \mathcal{X})$. Then $(a,x)\times_{m} (b,0) = (ab, xb)=(0,0) \ (b \in \mathcal{A})$. Hence $a=0$ as $\mathcal{A}$ is faithful and $x=0$ as $\mathcal{X}$ is $\mathcal{A}$-faithful.
\end{proof}

\begin{thm}
	Let $(\mathcal{A}, \| \cdot \|)$ be a Banach algebra and $\mathcal X$ be a Banach $\mathcal{A}$-bimodule. Then $N(\mathcal{A} \times_m \mathcal{X})$ is singleton if and only if $N(\mathcal{A})$ and $N(\mathcal X)$ are singleton.
\end{thm}

\begin{proof}
	Assume that $N(\mathcal{A})$ and $N(\mathcal X)$ are singleton. Set $J_1=\{(0,0)\}$ and $J_2= \{0\} \times \mathcal{X}$. Then $J_1$ and $J_2$ are closed ideals in $\mathcal{A} \times_{m} \mathcal{X}, \, J_1 \subset J_2, \, (\mathcal{A} \times_{m} \mathcal{X})/J_2 \cong \mathcal{A}$, and $J_2/J_1 \cong \mathcal{X}$. Since $N(\mathcal{A})$ and $N(\mathcal{X})$ are singleton, $N((\mathcal{A} \times_{m} \mathcal{X})/J_2)$ and $N(J_2/ J_1)$ are singleton. Hence, by Proposition~\ref{1}, $\mathcal{A} \times_{m} \mathcal X$ has unique algebra norm, i.e., $N(\mathcal{A} \times_{m} \mathcal X)$ is singleton. The proof of ``converse part" is similar to the proof of Theorem~\ref{uniqueness norm on direct-sum product}.
\end{proof}

\begin{defn}\label{Defn:Null}
	Let $(\mathcal{A}, \| \cdot \|)$ be a normed algebra, and set $\mathcal{A} \times_0 \mathbb{C}$, with product $(a, \alpha) \times_0 (b, \beta) = ab$ and $\|(a, \alpha)\| = \|a\| + | \alpha| \ ((a, \alpha) \in \mathcal{A} \times_0 \mathbb{C})$. Then $(\mathcal{A} \times_0 \mathbb{C}, \times_0, \| \cdot \|)$ is a normed algebra with rad($\mathcal{A} \times_0 \mathbb{C}$) = $\mathbb{C}$.
\end{defn} 

\noindent The next result is a generalization of  given example in~\cite[P. 633]{DL:97}.

 \begin{thm}\label{Norms on null product algebra}
 	Let $\mathcal{A}$ be a normed algebra with the co-dimension of $\mathcal{A}^2$ being infinite. Then $N(\mathcal{A} \times_0 \mathbb{C})$ is infinite. Moreover, if $N_c(\mathcal{A})$ is non-empty, then $N_c(\mathcal{A} \times_0 \mathbb{C})$ is infinite.
 \end{thm}
 
 \begin{proof} Let $(\mathcal{A}, \| \cdot \|)$ be a normed algebra. Since $\mathcal{A}^2$ has infinite co-dimension in $\mathcal{A}$, there exist infinite linearly independent subset $L$ of $\mathcal{A}$ such that $\mathcal{A}^2 \cap L = \phi$ and $\| a \| =1 \; (a \in L)$. For each $n \in \mathbb{N}$, choose $L_n = \{ a_{n1}, a_{n2}, \ldots \} \subset L$ such that: 
 	\begin{enumerate}
 		\item Each $L_n$ is infinite;
 		\item $L_n \cap L_m = \phi \ (n \neq m)$;
 		\item  $L = \bigcup_{n=1}^{ \infty} L_n$.
 	\end{enumerate}
 	Let $B_n$ be a (Hamel) basis of $\mathcal{A}$ such that $L_n \subset B_n$ for each $n \in \mathbb{N}$. Then each $C_n = B_n \cap \mathcal{A}^2$ is a basis of $\mathcal{A}^2$. Consider the (unique) linear map $\varphi_n : \mathcal{A} \longrightarrow \mathbb{C}$ such that
 	\[
 	\varphi_n(a)=
 	\begin{cases}
 		k&  (\text{if } a=a_{nk}  \in L_n \setminus \{a_{n1}\});\\
 		1 & (\text{if } a \in B_n \setminus (L_n \cup C_n)); \\
 		0 & (\text{if } a \in C_n \cup \{a_{n1}\}).
 	\end{cases}
 	\]
 	Next take $\mathcal{B} = \mathcal{A} \times_0 \mathbb{C}$. For $(a, \alpha), (b, \beta) \in \mathcal{B}$, define
 	$$(a, \alpha) (b, \beta) = (ab, 0) \quad \text{ and } \quad p((a, \alpha))=\|a\|+ |\alpha|.$$
 	Then $(\mathcal{B}, p( \cdot ))$ is a normed algebra. For each $n \in \mathbb{N}$, define
 	$$p_n((a, \alpha)) = \|a\| + |\varphi_n(a)- \alpha| \quad ((a, \alpha) \in \mathcal{A}).$$
 	Clearly, each $p_n(\cdot)$ is a linear norm. Let $(a, \alpha), (b, \beta) \in \mathcal{B}$. Then $p_n((a, \alpha)(b, \beta)) = p_n((ab, 0)) = \|ab\| \leq p_n((a, \alpha))p_n((b, \beta))$ because $ab \in \mathcal{A}^2$ and hence $\varphi_n(ab) = 0$. Thus each $p_n(\cdot) \in N(\mathcal{B})$. Now, we \textit{claim} that these norms are non-equivalent. Let $m < n$ and $g_k = a_{mk} \ (k \in \mathbb{N})$. Then, for $k \geq 2, \; p_m((a_k, 0)) = 1+k$ and $p_n((a_k, 0)) \leq 2$ because $\varphi_n(a_{mk}) = 0 \text{ or 1}$. Thus we have proved our claim. Hence $N_c(\mathcal{A})$ is an infinite set. Suppose $(\mathcal{A}, \| \cdot \|)$ is a Banach algebra. Finally, we \textit{claim} that each $p_n(\cdot) \in N_c(\mathcal{B})$. Let $(a_n, \alpha_n)$ be a Cauchy sequence in $(\mathcal{B}, p_n( \cdot ))$. Then $a_n$ is a Cauchy sequence in $(\mathcal{A}, \| \cdot \|)$, converges to $a$ and $\alpha_n$ is a Cauchy sequence in $\mathbb{C}$, converges to $\alpha$. Hence $(a_n, \alpha_n)$ converges to $(a, \alpha) \in \mathcal{B}$.
 \end{proof}

\noindent The next result is a direct application of Proposition~\ref{1} and providing sufficient condition for the co-dimension of $\mathcal{A}^2$ is finite in $\mathcal{A}$.

\begin{cor}\label{3}
	Let $\mathcal{A}$ be an algebra such that $N(\mathcal{A})$ and $N_c(\mathcal{A})$ are singleton. Then 
\begin{enumerate}[{(i)}]
\item $N(\mathcal{A} \times_0 \mathbb{C})$ is singleton.
\item the co-dimension of $\mathcal{A}^2$ is finite in $\mathcal{A}$.
\item if $\mathcal{A}^2$ is dense in $\mathcal{A}$, then $\mathcal{A}^2 = \mathcal{A}$.
\end{enumerate}
\end{cor}

\begin{proof}
$(i)$ Since $(\mathcal{A} \times_0 \mathbb{C})/\{0\} \times_0 \mathbb{C} \cong \mathcal{A}$ and $\{0\} \times_0 \mathbb{C}/ \{0\} \times_0 \{0\}$ have unique algebra norms. Hence by Proposition~\ref{1}, $(\mathcal{A} \times_0 \mathbb{C})/\{0\} \times_0 \{0\} \cong \mathcal{A} \times_0 \mathbb{C}$ has a unique algebra norm.\\
$(ii)$ Suppose if possible the co-dimension of $\mathcal{A}^2$ is infinite in $\mathcal{A}$. Then by above Theorem~\ref{Norms on null product algebra}, the Banach algebra $(\mathcal{A} \times_0 \mathbb{C}, p(\cdot))$ has infinitely many norms, but by Statement (i), the Banach algebra $(\mathcal{A} \times_0 \mathbb{C}, p(\cdot))$ has a unique norm. Hence the co-dimension of $\mathcal{A}^2$ is finite in $\mathcal{A}$.\\
$(iii)$ Suppose the co-dimension of $\mathcal{A}^2$ is $n$ in $\mathcal{A}$ for some $n \in \mathbb{N}$. Then the set $B= \{v_1 + \mathcal{A}^2, v_2 + \mathcal{A}^2 \ldots, v_n+ \mathcal{A}^2 \}$ is a basis of $\mathcal{A}/\mathcal{A}^2$. Now consider $W = \text{span}\{\mathcal{A}^2, v_1, v_2, \ldots, v_{n-1} \}$ is a hyperspace of $\mathcal{A}$, i.e. dim($\mathcal{A}/W)=1$. Then there exist discontinuous linear functional $\varphi : \mathcal{A} \longrightarrow \mathbb{C}$ such that ker$\varphi= W$. Define $\|| a \|| = \|a \| + |\varphi(a)| \; \; (a \in \mathcal{A})$ and $\| a \| \leq \|| a \|| \  (a \in \mathcal{A})$. Hence $n$ must be 0 and that gives $\mathcal{A}^2 = \mathcal{A}$.
\end{proof}

\section{Examples}

\begin{exam}
	Let $\mathcal{A} = M_2(\mathbb C)$ and ideal $\mathcal{I} = \left\langle \begin{bmatrix}
		a & 0 \\
		0 & 0 
	\end{bmatrix}, \begin{bmatrix}
		0 & b \\
		0 & 0 
	\end{bmatrix} \right\rangle$. Then $\mathcal{I}$ is a right non-faithful ideal of a faithful algebra $M_2(\mathbb{C})$.
\end{exam}

\begin{exam}
	Let $1 \leq p < \infty$ and let $\mathcal{A} = \ell^p$ with pointwise product. Then $N(\ell^p)>1$ and $(\ell^p)^2 = \ell^{\frac{p}{2}}$ is dense in $(\ell^p, \| \cdot \|_p)$ but $\ell^{\frac{p}{2}} \subsetneq \ell^p$. This example says that we can not relaxe the condition ``$N(\mathcal{A})$ is singleton'' from Corollary~\ref{3}$(iii)$.
\end{exam}

\begin{exam}
	Every ideal of a commutative semisimple algebra is faithful.
\end{exam}

\section{Question}

\noindent It would be interesting to examine whether  Proposition~\ref{1} hold without assuming completeness.

\section*{Declarations}

\noindent\textbf{Funding } The second author is very thankful to the University Grants Commission (UGC), New Delhi, for providing Senior Research Fellowship.\\

\noindent\textbf{Conflict of interest } The authors have no relevant financial or non-financial interests to disclose.

\end{document}